\documentclass[preprint,12pt]{elsarticle}

\usepackage{amssymb, amsthm, amsmath,enumitem,mathabx}
\usepackage{float}
\usepackage[colorlinks=true]{hyperref}

\newtheorem{theorem}{Theorem}[section]
\newtheorem{lemma}[theorem]{Lemma}
\newtheorem{corollary}[theorem]{Corollary}
\newtheorem{proposition}[theorem]{Proposition}
\newtheorem{remark}[theorem]{Remark}
\newtheorem{definition}[theorem]{Definition}

\newcommand{\R}{\mathbb R}

\newcommand{\eps}{\varepsilon}

\newcommand{\dif}{\ensuremath{\,\mathrm{d}}}

\newcommand{\vol}{\mathop{\mathrm{vol}}\nolimits}

\renewcommand{\div}{\mathop{\mathrm{div}}\nolimits}
\newcommand{\curl}{\mathop{\mathrm{curl}}\nolimits}

\numberwithin{equation}{section}

\begin{document}
	
	\begin{frontmatter}
		
		
		
		\title{A small rigid body immersed in a three-dimensional viscous fluid and the vanishing viscosity limit}

		\author[label2]{Xiaoguang You\corref{cor1}}
			\cortext[cor1]{Corresponding author. Email:xgyou@jxstnu.edu.cn}
		
		\affiliation[label2]{organization={School of Mathematical Sciences, Jiangxi Science $\&$ Technology Normal University},
			city={Nanchang},
			state={Jiangxi},
			country={China}}
		
		\begin{abstract}
In this paper, we investigate the asymptotic behavior of a coupled fluid--rigid body system, consisting of a rigid body immersed in a three-dimensional viscous incompressible fluid, as the body size and the
fluid viscosity vanish. We prove that the limiting velocity field is governed by the whole-space incompressible Euler equations. Moreover, the influence of the fluid on the rigid body motion is shown to vanish in this
limit.
		\end{abstract}

		\begin{keyword}
			Fluid-rigid interaction \sep Navier--Stokes equations \sep Euler equations \sep Asymptotic behavior
			
			
			\MSC 35B40 \sep 35Q30 \sep 35Q31
			
		\end{keyword}
		
	\end{frontmatter}
	
\section{Introduction}\label{sec1}

In this work, we study the interaction between a viscous incompressible fluid and a small rigid body completely immersed in the fluid. The fluid motion is governed by the three-dimensional Navier--Stokes equations, while
the rigid body motion is described by Newton's laws of motion, namely, the balance of linear and angular momentum. Our goal is to investigate the asymptotic behavior of the coupled fluid--rigid body system as the size of the rigid body and the viscosity coefficient tend to zero.

We now introduce the mathematical formulation of the problem. Let
$\Omega$ be a smooth, bounded, simply connected domain in $\mathbb R^3$.
For $\eps>0$, we define the scaled rigid body
\[
\Omega_\eps=\eps\Omega,
\]
and denote by
\[
\mathcal F_\eps=\mathbb R^3\setminus
\overline{\Omega_\eps}
\]
the corresponding fluid domain. For $t\ge0$, let
\[
\Omega_{\nu,\eps}(t)\subset\mathbb R^3
\]
	denote the region occupied by the rigid body, with \(\Omega_{\nu,\eps}(0)=\Omega_\eps\).  The fluid occupies the exterior domain
\[
\mathcal F_{\nu,\eps}(t)
=
\mathbb R^3\setminus
\overline{\Omega_{\nu,\eps}(t)} .
\]
We denote by $n^{\nu,\eps}$ the unit
normal vector on $\partial\Omega_{\nu,\eps}(t)$ pointing toward the
interior of the rigid body. The fluid density is normalized to one, and the viscosity coefficient is
denoted by $\nu>0$. The fluid velocity and pressure are denoted by $u^{\nu,\eps}$ and $p^{\nu,\eps}$, respectively. With the above notation, the coupled fluid--rigid body system for
$t>0$ is given by
\begin{itemize}[leftmargin=1em] 
	\item \textbf{Navier--Stokes equations:}
\begin{align}
	\partial_t u^{\nu,\eps} - \nu\Delta u^{\nu,\eps} + (u^{\nu,\eps} \cdot \nabla) u^{\nu,\eps} + \nabla p^{\nu,\eps} = 0, \quad  x \in \mathcal{F}_{\nu,\eps}(t),\label{Equ:ns-1}\\
	\div u^{\nu,\eps} = 0, \quad  x \in \mathcal{F}_{\nu,\eps}(t);
\end{align}
 \item \textbf{Newton's laws of motion:}
\begin{align}
	m^{\nu,\eps}\ddot{h}^{\nu,\eps}(t) = -\int_{\partial\Omega_{\nu,\eps}(t)}\sigma(u^{\nu,\eps}, p^{\nu,\eps}) n^{\nu,\eps} \dif x, \\
	\frac{\dif}{\dif t}
	\big(J^{\nu,\eps}(t)\omega^{\nu,\eps}(t)\big) = -\int_{\partial \Omega_{\nu,\eps}(t)} \big(x - h^{\nu,\eps}(t)\big) \times \big(\sigma(u^{\nu,\eps}, p^{\nu,\eps}) n^{\nu,\eps}\big) \dif x;
\end{align}

\item \textbf{Boundary conditions:}
\begin{align}
	u^{\nu,\eps}(x, t) =  \dot{h}^{\nu,\eps}(t) + \omega^{\nu,\eps}(t) \times \big(x - h^{\nu,\eps}(t)\big), \quad  x \in \partial \Omega_{\nu,\eps}(t),\\
	\lim_{|x|\rightarrow \infty } u^{\nu,\eps}(x, t) = 0;
\end{align}
\item \textbf{Initial conditions:}
\begin{equation}\label{Equ:ns-7}
	\begin{split}
		u^{\nu,\eps}(x, 0) = u_0^{\nu,\eps}(x), \quad x \in \mathcal{F}_\eps, \\
		h^{\nu,\eps}(0) = 0, \quad \dot{h}^{\nu,\eps}(0) = \ell_0 ,  \quad \omega^{\nu,\eps}(0) = \omega_0.
	\end{split}
\end{equation}

\end{itemize}
Here $h^{\nu,\eps}(t)$ denotes the trajectory of the center of mass of the rigid body, and $\omega^{\nu,\eps}(t)$ denotes its angular velocity. The quantities $m^{\nu,\eps}$ and $J^{\nu,\eps}$ denote the mass and the inertia matrix of the rigid body, respectively. They are defined by
\begin{align*}
	m^{\nu,\eps} =  \rho_s^{\nu,\eps} \vol(\Omega_\eps), \\
	(J^{\nu,\eps})_{ij} = \rho_s^{\nu,\eps}  \int_{\Omega_{\nu,\eps}(t)}\big[|x-h^{\nu,\eps}(t)|^2\delta_{ij} - \big(x_i-h^{\nu,\eps}(t)\big)\big(x_j-h^{\nu,\eps}(t)\big)\big] \dif x,
\end{align*}
where $i,j\in\{1,2,3\}$, $\rho_s^{\nu,\eps}$ is the density of the rigid body, $\vol(\Omega_\eps)$ its volume, and $\delta_{ij}$ the Kronecker delta. For simplicity, we assume that the density $\rho_s^{\nu,\eps}$ is
spatially constant in the rigid body $\Omega_{\nu,\eps}(t)$. The Cauchy stress tensor of the fluid is given by
\begin{align*}
	\sigma(u^{\nu,\eps}, p^{\nu,\eps}) = -p^{\nu,\eps}\mathbb{I}_3 + 2\nu D(u^{\nu,\eps}), 
\end{align*}
where $\mathbb{I}_3$ is the identity matrix in $\mathbb R^3$ and $D(u^{\nu,\eps})$ is the deformation tensor
\begin{equation*}\label{Equ:ns-5}
	D(u^{\nu,\eps}) = \frac{1}{2}\big(\nabla u^{\nu,\eps} + (\nabla u^{\nu,\eps})^\top\big). 
\end{equation*}

In this paper, the analysis is carried out in the framework of Leray--Hopf weak solutions
to the coupled fluid--rigid body system. To rigorously define this class of solutions, we first extend the velocity field $u^{\nu,\eps}$ to the whole space $\mathbb R^3$ (still denoted by $u^{\nu,\eps}$) by setting
\begin{equation}\label{Def:general-u}
	{u}^{\nu,\eps}(x, t) = \dot{h}^{\nu,\eps}(t) + \omega^{\nu,\eps}(t)\times(x - h^{\nu,\eps}(t)) \text{ for } x \in \Omega_{\nu,\eps}(t).
\end{equation}
Similarly, we extend the initial velocity field globally by defining
\begin{equation*}
	u_0^{\nu,\eps}(x) =\ell_0 + \omega_0\times x,   \text{ for }x \in {\Omega_{\eps}}.
\end{equation*}
The global density of the fluid--rigid system in $\mathbb{R}^3$ is defined as
\begin{equation*}
	{\rho}^{\nu,\eps}(x, t) = 1 \hbox{ for } x \in \mathcal{F}_{\nu,\eps}(t) \hbox{ and } {\rho}^{\nu,\eps} = \rho_s^{\nu,\eps} \hbox{ for } x \in {\Omega_{\nu,\eps}(t)},
\end{equation*}
and the initial global density is denoted by
\begin{equation}
	\rho^{\nu, \eps}_0(x) = 1 \text{ for } x \in \mathcal{F}_{\eps} \text{ and } \rho^{\nu, \eps}_0(x) = \rho_s^{\nu,\eps} \text{ for } x \in \Omega_\eps.
\end{equation}

We now define weak solutions to the system \eqref{Equ:ns-1}-\eqref{Equ:ns-7}.
\begin{definition}[See \cite{conca2000existence,desjardins2000weak,gunzburger2000global}] \label{Def:weak-solution} Suppose that the initial velocity $u_0^{\nu,\eps} \in L^2(\R^3)$ satisfies
	\begin{equation}\label{Initial-condition}
		\begin{cases}
			 \div u_0^{\nu,\eps} = 0 &\hbox{ in } \mathcal{F}_\eps, \\
			u_0^{\nu,\eps}(x) \cdot n^{\eps}(x) = (\ell_0 + \omega_0 \times x) \cdot n^{\eps}(x) &\hbox{ on } \partial\Omega_\eps,
		\end{cases}
	\end{equation}
	where $n^\eps$ is the unit normal vector field on $\partial\Omega_\eps$. A triplet $({u}^{\nu,\eps}, h^{\nu,\eps}, \omega^{\nu,\eps})$ is called a weak solution of the system \eqref{Equ:ns-1}-\eqref{Equ:ns-7}, if for every $T>0$,
		\begin{equation}
			\begin{split}
				{u}^{\nu,\eps} \in L^\infty(0, T; L^2(\mathbb{R}^3)) \text{ and } D(u^{\nu,\eps}) \in L^2(0, T;L^2(\mathbb{R}^3\times \R^3)),\\
				{u}^{\nu,\eps}(x, t) = \dot{h}^{\nu,\eps}(t) + \omega^{\nu,\eps}(t) \times \big(x - h^{\nu,\eps}(t)\big)   \text{ for } x \in \Omega_{\nu,\eps}(t), t\in(0, T),
			\end{split}
		\end{equation}
	and for every divergence-free vector field $\varphi^{\nu,\eps}\in C^1([0,T];H^1(\mathbb R^3))$
	such that $D(\varphi^{\nu,\eps}) = 0$ in $\Omega_{\nu,\eps}(t)$, the following identity
	
\begin{equation}\label{Equ:weak}
	\begin{aligned}
		&\left(\int_{\mathbb R^3}
		\rho^{\nu,\eps}u^{\nu,\eps}\cdot
		\varphi^{\nu,\eps}\,\dif x\right)(t) +2\nu
		\int_0^t\int_{\mathbb R^3}
		D(u^{\nu,\eps}):D(\varphi^{\nu,\eps})
		\dif x\dif s\\
		&=\int_{\mathbb R^3}
		\rho_0^{\nu,\eps}u_0^{\nu,\eps}\cdot
		\varphi^{\nu,\eps}(0)\dif x
	+\int_0^t\int_{\mathbb R^3}
		\rho^{\nu,\eps}u^{\nu,\eps}
		\cdot
		\big(
		\partial_s\varphi^{\nu,\eps}
		+
		(u^{\nu,\eps}\cdot\nabla)
		\varphi^{\nu,\eps}
		\big)
		\dif x\dif s
	\end{aligned}
\end{equation}
holds for almost all $t \in (0, T)$.  For matrices
$A,B\in\mathbb R^{3\times3}$, we use the notation
\[
A:B=\sum_{i,j=1}^{3}A_{ij}B_{ij}.
\]
\end{definition}
The existence of weak solutions to the system has been established in \cite{serre1987chute,conca2000existence,desjardins1999existence,gunzburger2000global}. Moreover, the triplet $(u^{\nu,\eps},h^{\nu,\eps},\omega^{\nu,\eps})$ is known to satisfy the following energy inequality:
\begin{equation}\label{Identity:energy}
	\begin{split}
		&\frac{1}{2}\|u^{\nu,\eps}\|_{L^2(\mathcal{F}_{\nu,\eps}(t))}^2 + E_{\mathrm{body}}^{\nu,\eps}(t)
		+2\nu\int_0^t \|D(u^{\nu,\eps})\|_{L^2(\mathcal{F}_{\nu,\eps}(s))}^2\,\dif s\\
		&\leqslant \frac{1}{2}\|u_0^{\nu,\eps}\|_{L^2(\mathcal{F}_\eps)}^2 + E_{\mathrm{body}}^{\nu,\eps}(0),
	\end{split}
\end{equation}
where $E_{\mathrm{body}}^{\nu,\eps}$ denotes the kinetic energy of the rigid body, given by
\begin{equation}
	E_{\mathrm{body}}^{\nu,\eps}(t)=\frac12\Bigl(m^{\nu,\eps}|\dot h^{\nu,\eps}|^2+(J^{\nu,\eps}\omega^{\nu,\eps})\cdot\omega^{\nu,\eps}\Bigr),
\end{equation}
and
\begin{equation}
	E_{\mathrm{body}}^{\nu,\eps}(0)=\frac12\Bigl(m^{\nu,\eps}|\ell_0|^2+(J^{\nu,\eps}_0\omega_0)\cdot\omega_0\Bigr).
\end{equation}
The matrix $J_0^{\nu,\eps}$ is defined by
\begin{equation*}
	(J^{\nu,\eps}_0)_{ij}=\rho_s^{\nu,\eps}\int_{\Omega_\eps}\bigl(\delta_{ij}|x|^2-x_ix_j\bigr)\,\dif x.
\end{equation*}
We now review some related results concerning the asymptotic behavior of
solutions in the shrinking-body limit. For a fixed rigid body immersed in the fluid, Iftimie et al. \cite{Iftimie2006two} proved for the two-dimensional Navier--Stokes equations that, as the size of the body tends to zero, the limit flow satisfies the full-plane Navier--Stokes equations. Later, similar conclusions for the 3D case were established by Iftimie and Kelliher \cite{Iftimie2009Remarks}. In \cite{Lacave2009two,Lacave2015three}, Lacave considered the case of a thin obstacle shrinking to a curve in two
and three dimensions, and it was proved that the limit flow satisfies the Navier--Stokes equations in the exterior of a curve.  For the moving rigid body case, where the body motion is governed by the
conservation laws of linear and angular momentum, the asymptotic behaviour of the fluid--rigid system becomes rather complicated. Lacave and Takahashi in \cite{lacave2017small} studied the asymptotic behaviour of a small moving disk immersed in a two-dimensional viscous incompressible fluid with small initial data, and obtained the convergence as the size tends to zero. Later, Feireisl et al. \cite{feireisl2023motion} extended the result by removing the restriction on the initial data and the shape of the rigid body for both two and three dimensions. More recently, He and Su \cite{he2025vanishing} further considered the case where the small rigid body shrinks to a ``massless'' point. They proved that, as the size of the rigid body tends to zero while its density remains unchanged, the \textit{Leray--Hopf} weak solution converges weakly in $L^2$ to a solution of the Navier--Stokes equations in the whole space.

Motivated by the above results, we investigate the asymptotic behavior of solutions of the system \eqref{Equ:ns-1}-\eqref{Equ:ns-7} as both the viscosity coefficient $\nu$ and the size parameter $\eps$ tend to zero. More precisely, let $\widebar u_0\in H^3(\mathbb R^3)$ be a divergence-free
vector field. It is known, see for example
\cite{ebin1970groups,kato1972non,majda2002vorticity}, that there exist a
maximal time $T^*>0$ and a unique solution
\begin{equation}
\widebar u\in C^1([0,T];H^2(\mathbb R^3))
\cap C([0,T];H^3(\mathbb R^3)),
\qquad 0<T<T^*,
\end{equation}
to the following Euler equations:
\begin{equation}
	\begin{split}
		\partial_t \widebar u + \widebar u \cdot \nabla \widebar u + \nabla \widebar p = 0,\\
		\div \widebar u = 0,\\
		\widebar u|_{t=0} = \widebar u_0.
	\end{split}
\end{equation}
The main question is whether $u^{\nu,\eps}$ converges to $\widebar u$, assuming $u^{\nu,\eps}_0\to\bar u_0$ as $\nu,\eps\to0$. Notably, for a fixed rigid body immersed in the fluid, Iftimie et al. \cite{iftimie2009incompressible} proved that, under the assumption  
\begin{equation}0 < \eps \leqslant k\nu
\end{equation}
for some sufficiently small constant $k$, the solution $u^{\nu,\eps}$ converges strongly to $\widebar u$ in $L^\infty(0, T; L^2(\mathbb{R}^3))$ as $\nu,\eps \rightarrow 0$. 

To extend the convergence result of \cite{iftimie2009incompressible} to our case, as usual in fluid--rigid interaction problems, a major difficulty is that the equations for the fluid hold in a time dependent domain, hence we have a free boundary value problem. A further difficulty is that the incompatibility between \(u^{\nu,\eps}\) and \(\widebar u\) at the boundary leads to the formation of a boundary layer. To overcome these difficulties, on the one hand, we construct  a suitable stream function to obtain a family of approximations $\{\widebar{u}^{\nu,\eps}\}$ for $\widebar u$, which vanish at the moving boundary. On the other hand, we apply the transformation method proposed in \cite{cumsille2008well} to construct  a boundary corrector term $\Lambda^{\nu,\eps}$, which coincides with the rigid body motion in a small neighborhood of the
boundary and reduces to the identity away from the boundary. Finally, by using the energy identity \eqref{Identity:energy}, we deduce an estimate for the error term $u^{\nu,\eps} - \widebar u$ and   establish the following main result.

\begin{theorem}\label{Thm:main}
		Let $0<T<T^*$ be fixed. Let $1<{\alpha} \leq {\beta} <3$. Suppose that  
	\begin{equation}
		\rho_s\eps^{-\alpha} \leq \rho_s^{\nu,\eps} \leq \rho_s\eps^{-\beta} \text{ for } \ 0 <\nu, \eps<1,
	\end{equation}
	where $\rho_s > 0$ is a constant. Assume that the extended initial velocity
	$u_0^{\nu,\eps}\in L^2(\mathbb R^3)$ satisfies \eqref{Initial-condition} and converges to $\widebar u_0$ in the way
	\begin{equation}\label{Assumption-initial}
		\| u_0^{\nu,\eps} - \widebar u_0\|_{L^2(\mathbb{R}^3)} \rightarrow 0 \hbox { as } \nu, \eps \rightarrow 0.
	\end{equation}
	Then there exists a constant $k=k(\Omega, \widebar u_0, T)>0$ such that, if $0 < \eps \leqslant k\nu$, then for all $t \in [0, T]$ we have
	\begin{equation}\label{Est:diff}
		\begin{split} 
		\| u^{\nu,\eps}(\cdot,t)& - \widebar u(\cdot,t)\|_{L^2} + \sqrt{E_{\mathrm{body}}^{\nu,\eps}(t)} \\
		&\leq K\big(\|u_0^{\nu,\eps} - \widebar u_0\|_{L^2(\mathbb{R}^3)} + \sqrt{\nu} + \eps^{{\min(3-\beta, \alpha-1)}/{2}}\big),
		\end{split}
	\end{equation}
where $K=K(\Omega,\widebar u_0,T,\rho_s)>0$ is independent of $\nu$ and $\eps$.
\end{theorem}

When $\nu$ and $\eps$ tend to zero, \eqref{Est:diff} implies that $u^{\nu,\eps}$ converges strongly to $\widebar u$ in the $L^2$-norm. Moreover, when the initial linear velocity $\ell_0$ vanishes, the critical finite-mass scaling case $\rho_s^{\nu,\eps}=\rho_s\eps^{-3}$ can also be treated.
\begin{theorem} \label{Thm:endpoint} Suppose that $\ell_0 = 0$ and $\rho_s^{\nu,\eps} = \rho_s\eps^{-3}$ for some constant $\rho_s >0$.  Then there exist constants $k=k(\Omega, \widebar u_0, T) > 0$ and $K=K(\Omega, \widebar u_0, T, \rho_s)$ such that
	\begin{equation}\label{Est:tmp-11111}
		\begin{split}
			&\| u^{\nu,\eps}(\cdot,t) - \widebar u(\cdot,t)\|_{L^2} + \sqrt{E_{\mathrm{body}}^{\nu,\eps}(t)} \leq K(\| u_0^{\nu,\eps} - \widebar u_0\|_{L^2(\mathbb{R}^3)} + \sqrt{\nu}),
		\end{split}
	\end{equation}
	for all $0 < \eps < k\nu$ and $0 \leq t \leq T$. 
\end{theorem}
\begin{remark}
	Under the assumption $1<\alpha\leq\beta<3$, we prove that the \textit{Leray--Hopf} weak solution converges strongly to the solution of the Euler equations. Moreover, the critical finite-mass scaling case $\rho_s^{\nu,\eps}=\rho_s\eps^{-3}$ can be achieved provided that the initial linear velocity of the rigid body vanishes. The bound $\beta\leq3$ is sharp: for $\beta>3$, the initial kinetic energy $E_{\mathrm{body}}^{\nu,\eps}(0)$ of the rigid body diverges as $\nu,\eps\to0$. The lower bound condition $\alpha>1$ is imposed
	to ensure that the boundary corrector constructed in the proof remains a
	small perturbation, which is necessary to compensate for the mismatch
	between $u^{\nu,\eps}$ and $\widebar u$ on the moving boundary.
	
	It is worth mentioning that, when $\nu>0$ is fixed, weak convergence was
	established in \cite{he2025vanishing} for the case
	$\alpha=\beta=0$ by using the $L^p$-$L^q$ estimates of the
	fluid--structure semigroup. However, this approach does not seem directly
	applicable to the present setting, since these estimates depend on the
	viscosity parameter $\nu$. Therefore, whether convergence can be obtained
	for $0\leq\alpha\leq 1$ remains an interesting open problem.
\end{remark}

The paper is organized as follows. In Section 2, we establish several auxiliary estimates and construct the boundary corrector required for the proof of the main results. In Section 3, we prove Theorems \ref{Thm:main} and \ref{Thm:endpoint}. In the final section, we discuss the assumption \eqref{Assumption-initial} and further consequences.

\section{Auxiliary Estimates and Boundary Correctors}

\begin{lemma} \label{Lm:m-j}
Suppose that the density $\rho_s^{\nu,\eps}$ of the rigid body satisfies the assumptions stated in Theorem \ref{Thm:main}. Then we have
	\begin{equation}\label{Est:mass}
		\begin{split}
			K_1\eps^{3-\alpha} \leq m^{\nu,\eps} \leq K_2\eps^{3-\beta},
		\end{split}
	\end{equation}
	and for all $v \in \mathbb{R}^3$,
	\begin{equation}\label{Est:inertia}
		K_1\eps^{5-\alpha}|v|^2 \leq (J^{\nu,\eps}v)\cdot v \leq K_2\eps^{5-\beta}|v|^2,
	\end{equation}
	where $K_1$ and $K_2$ are constants independent of $\nu$ and $\eps$.
\end{lemma}
\begin{proof}
	Notice that
	\begin{equation}
		m^{\nu,\eps}
		=
		\rho_s^{\nu,\eps}\vol(\Omega_\eps)
		=
		\rho_s^{\nu,\eps}\eps^3\vol(\Omega),
	\end{equation}
	which, together with the assumptions on $\rho_s^{\nu,\eps}$, yields
	\eqref{Est:mass}.
	
	We next prove \eqref{Est:inertia}. Let
	$R^{\nu,\eps}(t)$ be the rotation matrix associated with the rigid body
	motion, which satisfies
	\begin{equation}\label{Def:R}
		\begin{cases}
			\dot R^{\nu,\eps}(t)
			=
			A(\omega^{\nu,\eps}(t))
			R^{\nu,\eps}(t),
			\qquad t\geq0,
			\\[1mm]
			R^{\nu,\eps}(0)=\mathbb I_3,
		\end{cases}
	\end{equation}
	where
	\begin{equation}
		A(\omega^{\nu,\eps})
		=
		\begin{pmatrix}
			0&-\omega_3^{\nu,\eps}&\omega_2^{\nu,\eps}\\
			\omega_3^{\nu,\eps}&0&-\omega_1^{\nu,\eps}\\
			-\omega_2^{\nu,\eps}&\omega_1^{\nu,\eps}&0
		\end{pmatrix}.
	\end{equation}
	 Since $A(\omega^{\nu,\eps})$ is skew-symmetric,  $R^{\nu,\eps}(t)$ is an orthogonal matrix. Using the change of variables
	\[
	x=\eps R^{\nu,\eps}(t)y+h^{\nu,\eps}(t),
	\]
	we obtain
	\begin{equation*}
			(J^{\nu,\eps}v)\cdot v
			=\rho_s^{\nu,\eps}\eps^5 v^\top R^{\nu,\eps}	I_\Omega  (R^{\nu,\eps})^\top v,
	\end{equation*}
where
	\[
	I_\Omega
	=
	\left(
	\int_{\Omega}
	(\delta_{ij}|y|^2-y_iy_j)\dif y
	\right)_{1\leq i,j\leq3}.
	\]
	Since $\Omega$ is a bounded open domain, the matrix $I_\Omega$ is positive
	definite. Consequently, there exist constants $c_\Omega,C_\Omega>0$ such
	that
	\[
	c_\Omega |w|^2
	\leq
	w^\top I_\Omega w
	\leq
	C_\Omega |w|^2, \quad w \in \R^3.
	\]
	Finally, applying the assumptions
	\[
	\rho_s\eps^{-\alpha}
	\leq
	\rho_s^{\nu,\eps}
	\leq
	\rho_s\eps^{-\beta},
	\]
	we conclude that
	\[
	K_1\eps^{5-\alpha}|v|^2
	\leq
	(J^{\nu,\eps}v)\cdot v
	\leq
	K_2\eps^{5-\beta}|v|^2 .
	\]
	This proves \eqref{Est:inertia}.
\end{proof}
Combining the energy inequality \eqref{Identity:energy} with the above lemma, we immediately obtain the following estimates.
\begin{corollary}\label{Cor:h}
Under the assumptions of Theorem \ref{Thm:main}, there exists a constant $K$ independent of $\nu$ and $\eps$ such that
	\begin{equation}\label{Est:dot-h}
		|\dot{h}^{\nu,\eps}(t)| \leq K\eps^{(\alpha-3)/2} \hbox{ and } |\omega^{\nu,\eps}(t)| \leq K\eps^{(\alpha-5)/2}, \quad \forall t\in [0, \infty).
	\end{equation}
\end{corollary}

Recall that $\widebar u$ denotes the smooth Euler solution in $\mathbb{R}^3$,
which does not vanish on the boundary $\partial\Omega_{\nu,\eps}(t)$.
To estimate $u^{\nu,\eps}-\widebar u$, we construct a family of approximations
$\{\widebar u^{\nu,\eps}\}$ of $\widebar u$ such that
$\widebar u^{\nu,\eps}$ vanishes on $\partial\Omega_{\nu,\eps}(t)$. Let $\widebar \Psi$ be an arbitrary stream function of $\widebar u$, satisfying \(\curl\widebar \Psi=\widebar u\), and denote
\begin{equation*}
	\widebar \Psi^{\nu,\eps}(x,t) = \widebar \Psi(x, t) - \widebar \Psi(h^{\nu,\eps}(t), t).
\end{equation*}

Let $R>0$ be such that the ball of radius $R$, centered at the origin,
contains $\Omega$. Let $\varphi$ be a smooth function on $\mathbb{R}^3$
such that $\varphi(x)\equiv0$ if $0\leq |x|\leq R+1$,
$\varphi\geq0$, and $\varphi\equiv1$ if $|x|>R+2$. We set
\[
\varphi^{\nu,\eps}(x,t)
=
\varphi\left(\frac{x-h^{\nu,\eps}(t)}{\eps}\right)
\]
and define $\widebar u^{\nu,\eps}$ by
\begin{equation}
	\widebar u^{\nu,\eps}
	=
	\curl(\varphi^{\nu,\eps}\widebar \Psi^{\nu,\eps}).
\end{equation}
It can be checked that the vector field $\widebar u^{\nu,\eps}$ is
divergence-free and vanishes in a neighborhood of the boundary
$\partial\Omega_{\nu,\eps}(t)$.
\begin{lemma}\label{Lm:boundary-1}
	There exists a constant $K>0$ such that, for any
	$0<\nu,\eps<1$ and $0\leq t\leq T$, we have the following estimates:
\begin{equation}\label{Est:approximation-boundary}
	\left\{
	\begin{aligned}
		\|\nabla\varphi^{\nu,\eps}
		\big(\widebar p-\widebar p(h^{\nu,\eps}(t),t)\big)\|_{L^2}
		+
		\|\nabla\varphi^{\nu,\eps}\times\widebar \Psi^{\nu,\eps}\|_{L^2}
		&\leqslant K\eps^{3/2},
		\\
		\|\widebar u^{\nu,\eps} - \varphi^{\nu,\eps}\widebar u\|_{L^2} + \|\widebar u^{\nu,\eps}-\widebar u\|_{L^2}
		&\leqslant K\eps^{3/2},
		\\
		\|\widebar u^{\nu,\eps}\|_{L^\infty}\|\nabla \widebar u^{\nu,\eps} - \nabla \widebar u\|_{L^2}
		&\leqslant K\eps^{1/2},
		\\
		\|\partial_t\varphi^{\nu,\eps}\widebar u
		+
		\partial_t\nabla\varphi^{\nu,\eps}\times\widebar \Psi^{\nu,\eps}
		+
		\nabla\varphi^{\nu,\eps}\times\partial_t\widebar \Psi^{\nu,\eps}
		\|_{L^2}
		&\leqslant K\eps^{\alpha/2-1},
		\\
		\|\nabla\partial_i\varphi^{\nu,\eps}\times\widebar \Psi^{\nu,\eps}
		+
		\partial_i\varphi^{\nu,\eps}\widebar u
		+
		\nabla\varphi^{\nu,\eps}\times\partial_i\widebar \Psi^{\nu,\eps}
		\|_{L^\infty}
		&\leqslant K\eps^{-1},  i=1,2,3 .
	\end{aligned}
	\right.
\end{equation}
	where the $L^2$-norm and $L^\infty$-norm are defined on the domain
	$\mathcal{F}_{\nu,\eps}(t)$.
\end{lemma}
\begin{proof} Let us denote	\begin{equation}\label{Def:C-eps-nu}
		\mathcal{C}_{\nu,\eps}(t) = \big\{ x \in \mathbb{R}^3 \ \big| \ (R+1)\eps < |x-h^{\nu,\eps}(t)| < (R+2)\eps\big\}.
	\end{equation}
	Recall that $\widebar u$, $\nabla \widebar u$ are uniformly bounded in $\mathbb{R}^3$ and $\widebar \Psi^{\nu,\eps}(h(t), t) = 0$, we have
	\begin{equation}
		\|\widebar \Psi^{\nu,\eps}\|_{L^\infty(\mathcal{C}_{\nu,\eps})} + \|\widebar p - \widebar p(h(t), t)\|_{L^\infty(\mathcal{C}_{\nu,\eps})} \leqslant K\eps.
	\end{equation} 
	Furthermore, from the definition of $\varphi^{\nu,\eps}$ and $\widebar \Psi^{\nu,\eps}$, we can see that $\nabla \varphi^{\nu,\eps}$ is supported in the closure of $C_{\nu,\eps}(t)$ and satisfies
	\begin{equation}
		\|\nabla \varphi^{\nu,\eps}\|_{L^\infty} \leqslant K\eps^{-1},
	\end{equation}
	which yields the first inequality in \eqref{Est:approximation-boundary} by combining with the fact that the Lebesgue measure of $\mathcal{C}_{\nu,\eps}$ is $\mathcal{O}(\eps^3)$.
	
	To consider the second inequality in \eqref{Est:approximation-boundary}, it suffices to rewrite $\widebar u^{\nu,\eps} - \widebar u$ as:
	\begin{equation}
	\widebar u^{\nu,\eps} - \widebar u = (\varphi^{\nu,\eps}-1)\widebar u + \nabla \varphi^{\nu,\eps}\times \widebar \Psi^{\nu,\eps}.
	\end{equation}
	It can be checked that each term on the right hand side of the above equation is uniformly bounded with respect to $\nu$ and $\eps$, and supported in the closure of $\mathcal{C}_{\nu,\eps}$. Hence the second inequality holds true. The third bound and the last bound follow by similar support and boundedness arguments, we omit the details. 
	
	To examine the second to last bound, we need to rewrite $\partial_t\varphi^{\nu,\eps}$ and $\partial_t \widebar\Psi^{\nu,\eps}$ as
	\begin{equation*}
			\partial_t \varphi^{\nu,\eps}(x, t) = - \frac{\dot{h}^{\nu,\eps}(t)}{\eps} \cdot(\nabla\varphi)\Big(\frac{x - h^{\nu,\eps}( t)}{\eps}\Big),
	\end{equation*}
	and
	\begin{equation*}
			\partial_t \widebar\Psi^{\nu,\eps}(x, t) =  (\partial_t \widebar\Psi)(x, t) -(\partial_t \widebar\Psi)(h^{\nu,\eps}(t), t) - (\nabla \widebar\Psi)(h^{\nu,\eps}(t), t) \cdot \dot{h}^{\nu,\eps}(t),
	\end{equation*}
	respectively.	Notice that $\partial_t \widebar \Psi^{\nu,\eps}$ can be viewed as a stream function of $\partial_t \widebar u$, hence
	\begin{equation*}
		\|(\partial_t \widebar\Psi)(x, t) -(\partial_t \widebar\Psi)(h^{\nu,\eps}(t), t)\|_{L^\infty(\mathcal{C}_{\nu,\eps})} \leq K\eps.
	\end{equation*}
	The second to last bound then follows by combining with Corollary \ref{Cor:h}.
\end{proof}

Recall that the fluid domain in the fluid--rigid system is time-dependent, and since the velocity $u^{\nu,\eps}$ does not vanish on $\partial\Omega_{\nu,\eps}$, integration by parts cannot be applied directly on the fluid domain. To circumvent this issue, we construct a boundary corrector $\Lambda^{\nu,\eps}$ that agrees with $u^{\nu,\eps}$ in $\Omega_{\nu,\eps}(t)$. To this end, we introduce, for $x\in\mathbb R^3$, the auxiliary quantities
\begin{equation}\label{Def:V}
	\Phi^{\nu,\eps}(x,t) = -\frac{1}{2}\dot{h}^{\nu,\eps}(t)\times\bigl(x-h^{\nu,\eps}(t)\bigr)
	+\frac{1}{2}\bigl|x-h^{\nu,\eps}(t)\bigr|^2\omega^{\nu,\eps}(t),
\end{equation}
and
\begin{equation}\label{Def:g}
	v^{\nu,\eps}(x,t) = \dot{h}^{\nu,\eps}(t)+\omega^{\nu,\eps}(t)\times\bigl(x-h^{\nu,\eps}(t)\bigr).
\end{equation}
With these at hand, we define the corrector $\Lambda^{\nu,\eps}$ by
\begin{equation}\label{Def:Lambda}
	\Lambda^{\nu,\eps}(x,t)=(1-\varphi^{\nu,\eps})v^{\nu,\eps}
	+\nabla\varphi^{\nu,\eps}\times\Phi^{\nu,\eps}.
\end{equation}

The following lemma collects the key properties of $\Lambda^{\nu,\eps}$.
\begin{lemma}\label{Lm:boundary-corrector}
	Assume that $0<\nu,\eps<1$, and let $\Lambda^{\nu,\eps}$ be defined by \eqref{Def:Lambda}. Then $\Lambda^{\nu,\eps}$ is solenoidal in $\mathbb R^3$ and satisfies
	\begin{equation}\label{Est:corrector-1}
		\Lambda^{\nu,\eps}(x,t)=
		\begin{cases}
			0, & x\in \mathbb R^3\setminus B\bigl(h^{\nu,\eps}(t),(R+2)\eps\bigr),\\[2mm]
			v^{\nu,\eps}(x,t), & x\in\Omega_{\nu,\eps}(t).
		\end{cases}
	\end{equation}
	Moreover, there exists a constant $K>0$, independent of $\nu$ and $\eps$, such that for $k=0,1$,
	\begin{equation}\label{Est:lambda}
	\|\nabla^k \Lambda^{\nu,\eps}\|_{L^2(\mathbb R^3)}\le K\eps^{\alpha/2-k},\qquad
	\|\nabla^k \Lambda^{\nu,\eps}\|_{L^\infty(\mathbb R^3)}\le K\eps^{\alpha/2-1-k}.
	\end{equation}
\end{lemma}

\begin{proof}
	The property \eqref{Est:corrector-1} and the solenoidality of $\Lambda^{\nu,\eps}$ follow directly from the definition. It remains to verify \eqref{Est:lambda}. To this end, we observe that both $\nabla\varphi^{\nu,\eps}$ and $1-\varphi^{\nu,\eps}$ are supported in the closure of $B(h^{\nu,\eps}(t), (R+2)\eps)$, where $|x-h^{\nu,\eps}|=O(\eps)$. Combining this observation with Corollary~\ref{Cor:h} yields the desired bounds and completes the proof.
\end{proof}

Before ending this section, we present a modified Poincar\'e inequality.

\begin{lemma}\label{Lm:poincare}
	There exists a constant $K>0$, independent of $\nu$ and $\eps$,  such that for every $W^{\nu,\eps}\in H^1_0(\mathcal{F}_{\nu,\eps}(t))$,
	\begin{equation}\label{Est:poincare}
		\|W^{\nu,\eps}\|_{L^2(\mathcal A_{\nu,\eps}(t))} \le K\eps\,\|\nabla W^{\nu,\eps}\|_{L^2(\mathcal A_{\nu,\eps}(t))},
	\end{equation}
	where $\mathcal A_{\nu,\eps}(t):=\mathcal{F}_{\nu,\eps}(t)\cap B(h^{\nu,\eps}(t),(R+2)\eps)$.
\end{lemma}

\begin{proof}
	Let $R^{\nu,\eps}(t)$ be the rotation matrix defined in \eqref{Def:R}. Under the change of variables $x=R^{\nu,\eps}(t)y+h^{\nu,\eps}(t)$, we have
	\begin{equation*}
		\int_{\mathcal A_{\nu,\eps}(t)} |W^{\nu,\eps}(x,t)|^2\,\dif x
		=
		\int_{\mathcal A_\eps} |W^{\nu,\eps}(R^{\nu,\eps}(t)y+h^{\nu,\eps}(t))|^2\,\dif y,
	\end{equation*}
	where $\mathcal A_\eps:=\mathcal{F}_\eps\cap B(0,(R+2)\eps)$. Since $R^{\nu,\eps}(t)$ is orthogonal, applying the classical Poincar\'e inequality (see Lemma~3 in \cite{iftimie2009incompressible}) on the fixed domain $\mathcal A_\eps$ yields
	\begin{equation*}
		\int_{\mathcal A_\eps} |W^{\nu,\eps}(R^{\nu,\eps}(t)y+h^{\nu,\eps}(t))|^2\,\dif y
		\le
		K\eps \int_{\mathcal A_\eps} |\nabla W^{\nu,\eps}(R^{\nu,\eps}(t)y+h^{\nu,\eps}(t))|^2\,\dif y.
	\end{equation*}
	Changing back to the original variable $x$ gives \eqref{Est:poincare}.
\end{proof}

\section{Proof of Theorem \ref{Thm:main} and \ref{Thm:endpoint}} 
Let $u^{\nu,\eps}$ and $\widebar u$ be as described in the introduction. Then we have
\begin{equation}\label{Est:tmp-100}
	\|{u}^{\nu,\eps} - \widebar u\|_{L^2(\mathbb{R}^3)} \leqslant \|{u}^{\nu,\eps} - \widebar u\|_{L^2(\Omega_{\nu,\eps}(t))} +\|{u}^{\nu,\eps} - \widebar u\|_{L^2(\mathcal{F}_{\nu,\eps}(t))}.
\end{equation}

The first term on the right hand side of \eqref{Est:tmp-100} is easy to handle. Indeed, we infer from \eqref{Def:general-u} and Corollary \ref{Cor:h} that $u^{\nu,\eps}$ is $\mathcal{O}(\eps^{(\alpha-3)/2})$ in $\Omega_{\nu,\eps}(t)$. Combining with the fact that the Lebesgue measure of $\Omega_{\nu,\eps}(t)$ is $\mathcal{O}(\eps^3)$, we deduce that
\begin{equation}\label{Est:tmp-1000001} 
	\|{u}^{\nu,\eps} - \widebar u\|_{L^2(\Omega_{\nu,\eps}(t))} \leqslant K\eps^{\alpha/2}.
\end{equation} 

We now focus on the second term on the right hand side of \eqref{Est:tmp-100}. To simplify notations, the $L^2$ and $L^\infty$ norms that appeared in the rest of this section,  unless otherwise specified,  are defined on the domain $\mathcal{F}_{\nu,\eps}(t)$. Observing that ${u}^{\nu,\eps}$ and $\widebar u$ do not match on the boundary $\partial\Omega_{\nu,\eps}(t)$, it is not feasible to estimate this term directly by using the energy method. To handle this issue, we here apply the approximation $\widebar u^{\nu,\eps}$ constructed in Lemma \ref{Lm:boundary-1} to obtain
\begin{equation}\label{Est:split}
	\begin{split}
		\|{u}^{\nu,\eps} - \widebar u\|_{L^2} \leqslant  \|\widebar u^{\nu,\eps} - \widebar u\|_{L^2} + \|{u}^{\nu,\eps} - \widebar u^{\nu,\eps}\|_{L^2}.
	\end{split}
\end{equation}
Lemma \ref{Lm:boundary-1} further tells that
\begin{equation}\label{Est:aa}
	\|\widebar u^{\nu,\eps} - \widebar u\|_{L^2} \leqslant K\eps^{3/2}.
\end{equation}

To estimate the second term on the right hand side of \eqref{Est:split}, we first notice that
\begin{equation}\label{Equ:bar-u}
	\partial_t \widebar u^{\nu,\eps} = -\varphi^{\nu,\eps}(\widebar u\cdot \nabla \widebar u + \nabla \widebar p) + \partial_t \varphi^{\nu,\eps} \widebar u + \partial_t \nabla \varphi^{\nu,\eps} \times \widebar \Psi^{\nu,\eps} + \nabla \varphi^{\nu,\eps} \times \partial_t \widebar \Psi^{\nu,\eps},
\end{equation}
which implies that the error term $W^{\nu,\eps} := {u}^{\nu,\eps} - \widebar u^{\nu,\eps}$ satisfies
\begin{equation}
	\begin{split}
		&\partial_t W^{\nu, \eps} - \nu \Delta W^{\nu,\eps} =  -u^{\nu,\eps}\cdot \nabla u^{\nu,\eps} - \nabla p^{\nu,\eps} + \nu \Delta\widebar{u}^{\nu,\eps}  \\
		&\qquad+ \varphi^{\nu,\eps}(\widebar u\cdot \nabla \widebar u + \nabla \widebar p) - \partial_t \varphi^{\nu,\eps} \widebar u - \partial_t \nabla \varphi^{\nu,\eps} \times \widebar \Psi^{\nu,\eps} - \nabla \varphi^{\nu,\eps}\times\partial_t \widebar \Psi^{\nu,\eps}.
	\end{split}
\end{equation}

However, we cannot directly estimate $W^{\nu,\eps}$ by multiplying the above equation by $W^{\nu,\eps}$ and integrating over $\mathcal{F}_{\nu,\eps}(t)$. This is because $u^{\nu,\eps}$ has weak regularity and does not vanish on the boundary, leading to ill-defined integrals of certain nonlinear and viscous terms. Fortunately, we can use a classical technique proposed in \cite{iftimie2009incompressible} to address this problem. In fact, multiplying the above equation by $W^{\nu,\eps}$  is equivalent to the following combination of operations: multiplying \eqref{Equ:ns-1} by $u^{\nu,\eps}$, and then subtracting both the product of \eqref{Equ:ns-1} with $\widebar u^{\nu,\eps}$ and the product of \eqref{Equ:bar-u} with $W^{\nu,\eps}$. More precisely, this equivalence is obtained as follows. On the one hand, \(u^{\nu,\eps}\) satisfies the energy inequality \eqref{Identity:energy}, which reads
\begin{equation}\label{Identity:energy-proof}
	\frac{1}{2}\|u^{\nu,\eps}\|_{L^2}^2 + E_{\mathrm{body}}^{\nu,\eps}(t)
	+2\nu\int_0^t \|D(u^{\nu,\eps})\|_{L^2}^2\,\dif s
	\le \frac{1}{2}\|u_0^{\nu,\eps}\|_{L^2}^2 + E_{\mathrm{body}}^{\nu,\eps}(0).
\end{equation}

On the other hand, since \(\widebar u^{\nu,\eps}\in C^1([0,T];H^1(\mathbb R^3))\) and \(\widebar u^{\nu,\eps}=0\) in \(\Omega_{\nu,\eps}(s)\) for \(0\le s\le T\), we may take \(\varphi^{\nu,\eps}=\widebar u^{\nu,\eps}\) in the weak formulation \eqref{Equ:weak} to get
\begin{equation}\label{Equ:weak-proof}
	\begin{split}
		&\int_{\mathcal{F}_{\nu,\eps}(t)} u^{\nu,\eps}\cdot \widebar u^{\nu,\eps}\,\dif x
		+2\nu\int_0^t\int_{\mathcal{F}_{\nu,\eps}(s)} D(u^{\nu,\eps}):D(\widebar u^{\nu,\eps})\,\dif x\dif s\\
		&\quad= \int_{\mathcal{F}_\eps} u_0^{\nu,\eps}\cdot \widebar u_0^{\nu,\eps}\,\dif x
		+\int_0^t\int_{\mathcal{F}_{\nu,\eps}(s)} u^{\nu,\eps}\cdot \bigl(\partial_s\widebar u^{\nu,\eps}
		+(u^{\nu,\eps}\cdot\nabla)\widebar u^{\nu,\eps}\bigr)\,\dif x\dif s.
	\end{split}
\end{equation}

Furthermore, since \(\widebar u^{\nu,\eps}\) is smooth, taking the inner product of \eqref{Equ:bar-u} with \(W^{\nu,\eps}\) yields
\begin{equation}\label{Equ:tmp}
	\begin{split}
		&\int_0^t\int_{\mathcal{F}_{\nu,\eps}(s)} \partial_s\widebar u^{\nu,\eps}\cdot W^{\nu,\eps}\,\dif x\dif s\\
		&\quad= -\int_0^t\int_{\mathcal{F}_{\nu,\eps}(s)} \varphi^{\nu,\eps}(\nabla\widebar p+\widebar u\cdot\nabla\widebar u)\cdot W^{\nu,\eps}\,\dif x\dif s\\
		&\qquad+\int_0^t\int_{\mathcal{F}_{\nu,\eps}(s)} \bigl[\partial_t\varphi^{\nu,\eps}\widebar u
		+\partial_t\nabla\varphi^{\nu,\eps}\times\widebar\Psi^{\nu,\eps}
		+\nabla\varphi^{\nu,\eps}\times\partial_t\widebar\Psi^{\nu,\eps}\bigr]\cdot W^{\nu,\eps}\,\dif x\dif s.
	\end{split}
\end{equation}
By the Reynolds transport theorem, the left-hand side of \eqref{Equ:tmp} can be rewritten as
\begin{equation}\label{Equ:rey}
	\begin{split}
		&\int_0^t\int_{\mathcal{F}_{\nu,\eps}(s)} \partial_s\widebar u^{\nu,\eps}\cdot W^{\nu,\eps}\,\dif x\dif s\\
		&\quad= \int_0^t\int_{\mathcal{F}_{\nu,\eps}(s)} \partial_s\widebar u^{\nu,\eps}\cdot u^{\nu,\eps}\,\dif x\dif s
		+\int_0^t\int_{\mathcal{F}_{\nu,\eps}(s)} (u^{\nu,\eps}\cdot\nabla\widebar u^{\nu,\eps})\cdot \widebar u^{\nu,\eps}\,\dif x\dif s\\
		&\qquad-\frac{1}{2}\|\widebar u^{\nu,\eps}\|_{L^2}^2+\frac{1}{2}\|\widebar u_0^{\nu,\eps}\|_{L^2}^2,
	\end{split}
\end{equation}
where $\bar u_0^{\nu,\eps} := \bar u^{\nu,\eps}(\cdot,0)$. Since $\bar u^{\nu,\eps}$ vanishes in a neighbourhood of $\partial\Omega_{\nu,\eps}(s)$, the boundary term arising from the Reynolds transport theorem vanishes. Then, subtracting \eqref{Equ:weak-proof} and \eqref{Equ:tmp} from \eqref{Identity:energy-proof} and applying \eqref{Equ:rey}, we obtain
\begin{equation}\label{Est:sum}
	\begin{split}
		&\frac{1}{2}\|W^{\nu,\eps}\|_{L^2}^2+E_{\mathrm{body}}^{\nu,\eps}(t)
		+2\nu\int_0^t\|D(W^{\nu,\eps})\|_{L^2}^2\,\dif s\\
		&\qquad\qquad\qquad\qquad= \frac{1}{2}\|u_0^{\nu,\eps}-\widebar u_0^{\nu,\eps}\|_{L^2}^2+E_{\mathrm{body}}^{\nu,\eps}(0)+\sum_{i=1}^4 I_i,
	\end{split}
\end{equation}
where
\begin{equation}
	\begin{split}
		I_1 &= -2\nu\int_0^t\int_{\mathcal{F}_{\nu,\eps}(s)} D(\widebar u^{\nu,\eps}):D(W^{\nu,\eps})\,\dif x\dif s,\\
		I_2 &= -\int_0^t\int_{\mathcal{F}_{\nu,\eps}(s)} \bigl(u^{\nu,\eps}\cdot\nabla\widebar u^{\nu,\eps}
		-\varphi^{\nu,\eps}\widebar u\cdot\nabla\widebar u\bigr)\cdot W^{\nu,\eps}\,\dif x\dif s,\\
		I_3 &= \int_0^t\int_{\mathcal{F}_{\nu,\eps}(s)} \varphi^{\nu,\eps}\nabla\widebar p\cdot W^{\nu,\eps}\,\dif x\dif s,\\
		I_4 &= -\int_0^t\int_{\mathcal{F}_{\nu,\eps}(s)} \bigl[\partial_t\varphi^{\nu,\eps}\widebar u
		+\partial_t\nabla\varphi^{\nu,\eps}\times\widebar\Psi^{\nu,\eps}
		+\nabla\varphi^{\nu,\eps}\times\partial_t\widebar\Psi^{\nu,\eps}\bigr]\cdot W^{\nu,\eps}\,\dif x\dif s.
	\end{split}
\end{equation}

Using Lemma \ref{Lm:m-j} and Lemma \ref{Lm:boundary-1}, we have the following estimate for the initial data:
\begin{equation}\label{Est:initial-data}
	\frac{1}{2}\|u_0^{\nu,\eps}-\widebar u_0^{\nu,\eps}\|_{L^2}^2+E_{\mathrm{body}}^{\nu,\eps}(0)
	\le \frac{1}{2}\|u_0^{\nu,\eps}-\widebar u_0\|_{L^2}^2+K\eps^{3-\beta}.
\end{equation}

Thus, the proof of Theorem \ref{Thm:main} and Theorem \ref{Thm:endpoint} reduces to establishing the following bounds for the remainders \(I_i\), which we temporarily assume:
\begin{equation}\label{Est:summary}
	\left\{
	\begin{split}
		I_1 &\le \nu\int_0^t\|D(W^{\nu,\eps})\|_{L^2}^2\,\dif s+K\nu,\\
		I_2 &\le K\Bigl(\eps\int_0^t\|D(W^{\nu,\eps})\|_{L^2}^2\,\dif s
		+\int_0^t\|W^{\nu,\eps}\|_{L^2}^2\,\dif s+\eps^{\min(1,\alpha-1)}\Bigr),\\
		I_3 &\le K\Bigl(\int_0^t\|W^{\nu,\eps}\|_{L^2}^2\,\dif s+\eps^3\Bigr),\\
		I_4 &\le K\Bigl(\eps\int_0^t\|D(W^{\nu,\eps})\|_{L^2}^2\,\dif s+\eps^{\alpha-1}\Bigr).
	\end{split}
	\right.
\end{equation}

Indeed, substituting \eqref{Est:initial-data} and \eqref{Est:summary} into \eqref{Est:sum} yields, for some constant \(K_1>0\),
\begin{equation}\label{Equ:diff-t-new}
	\begin{split}
		&\frac{1}{2}\|W^{\nu,\eps}\|_{L^2}^2+E_{\mathrm{body}}^{\nu,\eps}(t)
		+(\nu-K_1\eps)\int_0^t\|D(W^{\nu,\eps})\|_{L^2}^2\,\dif s\\
		&\quad\le \frac{1}{2}\|u_0^{\nu,\eps}-\widebar u_0\|_{L^2}^2
		+K\Bigl(\int_0^t\|W^{\nu,\eps}\|_{L^2}^2\,\dif s
		+\nu+\eps^{\min(1,\alpha-1,3-\beta)}\Bigr).
	\end{split}
\end{equation}

Choosing \(0<k<K_1^{-1}\) and assuming \(0<\eps\le k\nu\), Gr\"onwall's inequality implies
\begin{equation}\label{Equ:diff-t-new-new}
	\|W^{\nu,\eps}\|_{L^2}+\sqrt{E_{\mathrm{body}}^{\nu,\eps}(t)}
	\le K\Bigl(\|u_0^{\nu,\eps}-\widebar u_0\|_{L^2}
	+\eps^{\min(\alpha-1,3-\beta)/2}+\sqrt{\nu}\Bigr).
\end{equation}
Combining \eqref{Est:tmp-100}-\eqref{Est:aa} with \eqref{Equ:diff-t-new-new} completes the proof of Theorem \ref{Thm:main}.

Although Theorem \ref{Thm:main} is stated under the assumption
\(
1<\alpha\leq\beta<3,
\)
all the estimates in Section 2 and Proposition 3.1 are valid in the critical finite-mass scaling case $\rho_s^{\nu,\eps}=\rho_s\eps^{-3}$ (i.e. \(\alpha=\beta=3\))
provided that \(\ell_0=0\). In fact, since $J_0^{\nu,\eps}=O(\eps^2)$ in this case, the estimate \eqref{Est:initial-data} can be improved to
\begin{equation}
	\frac{1}{2}\|u_0^{\nu,\eps}-\widebar u_0^{\nu,\eps}\|_{L^2}^2+E_{\mathrm{body}}^{\nu,\eps}(0)
	\le \|u_0^{\nu,\eps}-\widebar u_0\|_{L^2}^2+K\eps^2.
\end{equation}
Consequently, 
\begin{equation}\label{Equ:diff-t-new-new-new}
	\|W^{\nu,\eps}\|_{L^2}+\sqrt{E_{\mathrm{body}}^{\nu,\eps}(t)}
	\le K\Bigl(\|u_0^{\nu,\eps}-\widebar u_0\|_{L^2}+\sqrt{\nu}\Bigr),
\end{equation}
which completes the proof of Theorem \ref{Thm:endpoint}.

It remains to verify the estimates in \eqref{Est:summary}. To this end, we first establish an auxiliary proposition.

\begin{proposition}\label{Pro:w}
	There exists a constant \(K>0\), independent of \(\nu\) and \(\eps\), such that for any \(v^{\nu,\eps}\in H^1(\mathcal{F}_{\nu,\eps}(t))\) satisfying \(v^{\nu,\eps}=\Lambda^{\nu,\eps}\) on \(\partial\Omega_{\nu,\eps}(t)\), we have
	\begin{equation}\label{Est:poin}
		\|v^{\nu,\eps}\|_{L^2(\mathcal A_{\nu,\eps}(t))}
		\le K\Bigl(\eps\|D(v^{\nu,\eps})\|_{L^2(\mathcal{F}_{\nu,\eps}(t))}+\eps^{\alpha/2}\Bigr),
	\end{equation}
	where \(\mathcal A_{\nu,\eps}(t)\) is defined in Lemma \ref{Lm:poincare}.
\end{proposition}

\begin{proof}
	Since \(v^{\nu,\eps}-\Lambda^{\nu,\eps}\in H_0^1(\mathcal{F}_{\nu,\eps}(t))\), Lemma \ref{Lm:poincare} yields
	\begin{equation}
		\begin{split}
			\|v^{\nu,\eps}-\Lambda^{\nu,\eps}\|_{L^2(\mathcal A_{\nu,\eps}(t))}
			&\le K\eps\|\nabla(v^{\nu,\eps}-\Lambda^{\nu,\eps})\|_{L^2(\mathcal A_{\nu,\eps}(t))}\\
			&\le K\eps\|D(v^{\nu,\eps}-\Lambda^{\nu,\eps})\|_{L^2(\mathcal{F}_{\nu,\eps}(t))},
		\end{split}
	\end{equation}
	where in the second inequality we used integration by parts. Combining this with Lemma \ref{Lm:boundary-corrector} immediately gives \eqref{Est:poin}.
\end{proof}

With this proposition at hand, we proceed to verify the bounds for \(I_i\), \(i=1,\dots,4\), in \eqref{Est:summary}.
\paragraph{\textbf{Part 1. The viscous term $I_1$}} 
From the Cauchy--Schwarz and Young inequalities, we have
\begin{equation*}
	I_1 \leqslant \nu\int_0^t\|D(W^{\nu,\eps})\|_{L^2}^2\dif s+ K\nu\int_0^t\|D(\widebar u^{\nu,\eps})\|_{L^2}^2\dif s.
\end{equation*}
Applying Lemma \ref{Lm:boundary-1} then yields
\begin{equation}
	I_1 \leqslant \nu\int_0^t\|D(W^{\nu,\eps})\|_{L^2}^2\dif s + K\nu.
\end{equation}

\paragraph{\textbf{Part 2. The nonlinear term $I_2$}}

We begin by splitting \(I_2\) into two parts:
\begin{equation}\label{Equ:I_2}
	\begin{split}
		I_2 
		&= -\int_0^t\int_{\mathcal{F}_{\nu,\eps}(s)} \bigl(W^{\nu,\eps}\cdot \nabla \widebar u^{\nu,\eps}\bigr)\cdot W^{\nu,\eps}\,\dif x\dif s\\
		&\quad - \int_0^t\int_{\mathcal{F}_{\nu,\eps}(s)} \bigl(\widebar u^{\nu,\eps}\cdot \nabla \widebar u^{\nu,\eps} - \varphi^{\nu,\eps}\widebar u\cdot \nabla \widebar u\bigr)\cdot W^{\nu,\eps}\,\dif x\dif s.
	\end{split}
\end{equation}

To estimate the first term on the right-hand side of \eqref{Equ:I_2}, we decompose \(\nabla \widebar u^{\nu,\eps}\) as
\begin{equation*}
	\begin{split}
		\nabla \widebar u^{\nu,\eps}
		&= \bigl(\nabla^2\varphi^{\nu,\eps}\times\widebar\Psi^{\nu,\eps}
		+\nabla\varphi^{\nu,\eps}\times\widebar u
		+\nabla\varphi^{\nu,\eps}\times\nabla\widebar\Psi^{\nu,\eps}\bigr)
		+\varphi^{\nu,\eps}\nabla\widebar u\\
		&=: g^{\nu,\eps}+\varphi^{\nu,\eps}\nabla\widebar u.
	\end{split}
\end{equation*}
Since \(\nabla\varphi^{\nu,\eps}\) is supported in \(\overline{\mathcal A_{\nu,\eps}(t)}\), the above decomposition yields
\begin{equation*}
	\begin{split}
		&-\int_0^t\int_{\mathcal{F}_{\nu,\eps}(s)} \bigl(W^{\nu,\eps}\cdot \nabla \widebar u^{\nu,\eps}\bigr)\cdot W^{\nu,\eps}\,\dif x\dif s \\
		&\quad\le \int_0^t \|W^{\nu,\eps}\|_{L^2}^2 \|\varphi^{\nu,\eps}\nabla\widebar u\|_{L^\infty}\,\dif s
		+ \int_0^t \|W^{\nu,\eps}\|_{L^2(\mathcal A_{\nu,\eps}(s))}^2 \|g^{\nu,\eps}\|_{L^\infty}\,\dif s.
	\end{split}
\end{equation*}
Applying Lemma \ref{Lm:boundary-1} and Proposition \ref{Pro:w} yields
\begin{equation}\label{Est:first}
	\begin{split}
		&-\int_0^t\int_{\mathcal{F}_{\nu,\eps}(s)} \bigl(W^{\nu,\eps}\cdot \nabla \widebar u^{\nu,\eps}\bigr)\cdot W^{\nu,\eps}\,\dif x\dif s\\
		&\quad\le K\Bigl(\eps\int_0^t\|D(W^{\nu,\eps})\|_{L^2}^2\,\dif s
		+\int_0^t\|W^{\nu,\eps}\|_{L^2}^2\,\dif s
		+\eps^{\alpha-1}\Bigr).
	\end{split}
\end{equation}

We now turn to the second term in \eqref{Equ:I_2}. Splitting the integrand gives
\begin{equation*}
	\begin{split}
		&-\int_0^t\int_{\mathcal{F}_{\nu,\eps}(s)} \bigl(\widebar u^{\nu,\eps}\cdot \nabla \widebar u^{\nu,\eps} - \varphi^{\nu,\eps}\widebar u\cdot \nabla \widebar u\bigr)\cdot W^{\nu,\eps}\,\dif x\dif s\\
		&= -\int_0^t\int_{\mathcal{F}_{\nu,\eps}(s)} \bigl(\widebar u^{\nu,\eps}\cdot \nabla(\widebar u^{\nu,\eps}-\widebar u)\bigr)\cdot W^{\nu,\eps}\,\dif x\dif s\\
		&\quad -\int_0^t\int_{\mathcal{F}_{\nu,\eps}(s)} \bigl((\widebar u^{\nu,\eps}-\varphi^{\nu,\eps}\widebar u)\cdot \nabla\widebar u\bigr)\cdot W^{\nu,\eps}\,\dif x\dif s.
	\end{split}
\end{equation*}
Then, by Lemma \ref{Lm:boundary-1}, we obtain
\begin{equation}\label{Est:third}
	\begin{split}
		&-\int_0^t\int_{\mathcal{F}_{\nu,\eps}(s)} \bigl(\widebar u^{\nu,\eps}\cdot \nabla \widebar u^{\nu,\eps} - \varphi^{\nu,\eps}\widebar u\cdot \nabla \widebar u\bigr)\cdot W^{\nu,\eps}\,\dif x\dif s \\
		&\le K\eps^{1/2}\int_0^t\|W^{\nu,\eps}\|_{L^2}\,\dif s.
	\end{split}
\end{equation}

Finally, substituting \eqref{Est:first} and \eqref{Est:third} into \eqref{Equ:I_2} and applying Young's inequality yields
\begin{equation}
	I_2 \le K\Bigl(\eps\int_0^t\|D(W^{\nu,\eps})\|_{L^2}^2\,\dif s
	+\int_0^t\|W^{\nu,\eps}\|_{L^2}^2\,\dif s
	+\eps^{\min(1,\alpha-1)}\Bigr).
\end{equation}

\paragraph{\textbf{Part 3. The remainders $I_3$ and $I_4$}} 
We first integrate by parts to obtain
\begin{equation}
	I_3 = \int_0^t\int_{\mathcal{F}_{\nu,\eps}(s)} \bigl(\widebar p(x,s)-\widebar p(h^{\nu,\eps}(s),s)\bigr)\nabla\varphi^{\nu,\eps}\cdot W^{\nu,\eps}\,\dif x\dif s.
\end{equation}
Invoking Lemma \ref{Lm:boundary-1} then gives
\begin{equation}
	I_3 \le K\eps^{3/2}\int_0^t\|W^{\nu,\eps}\|_{L^2}\,\dif s
	\le K\Bigl(\int_0^t\|W^{\nu,\eps}\|_{L^2}^2\,\dif s+\eps^3\Bigr).
\end{equation}

It remains to estimate \(I_4\). The expression
\[
\partial_t\varphi^{\nu,\eps}\widebar u+\partial_t\nabla\varphi^{\nu,\eps}\times\widebar\Psi^{\nu,\eps}
+\nabla\varphi^{\nu,\eps}\times\partial_t\widebar\Psi^{\nu,\eps}
\]
is supported in \(\overline{\mathcal A_{\nu,\eps}(t)}\). Hence, by Lemma \ref{Lm:boundary-1} and Proposition \ref{Pro:w},
\begin{equation}
	I_4 \le K\Bigl(\eps^{\alpha/2}\|D(W^{\nu,\eps})\|_{L^2}+\eps^{\alpha-1}\Bigr).
\end{equation}
Applying Young's inequality then yields the bound stated in \eqref{Est:summary}.

\section{Conclusions}
Note that Theorem \ref{Thm:main} imposes a convergence assumption on the initial data. In this section, we analyze this assumption and draw relevant conclusions.

\begin{proposition}\label{Pr:app}
	Let $\widebar u_0\in H^3(\mathbb R^3)$ be divergence-free. Then there exists a family of approximations $u_0^{\nu,\eps}\in L^2(\mathbb R^3)$ satisfying \eqref{Initial-condition} and the estimate
	\begin{equation}\label{Est:initial}
		\|u_0^{\nu,\eps}-\widebar u_0\|_{L^2(\mathbb R^3)}\le K\eps.
	\end{equation}
\end{proposition}

\begin{proof}
	Let $\Psi$ be a stream function for $\widebar u_0$, normalized so that $\Psi(0)=0$. Let $\varphi\in C^\infty(\mathbb R^3)$ be a cut-off function such that $\varphi\equiv 0$ for $0\le |x|\le R+1$, $\varphi\ge 0$, and $\varphi\equiv 1$ for $|x|\ge R+2$. Set $\varphi^\eps(x):=\varphi(x/\eps)$. For $x\in\mathbb R^3$, define
	\begin{equation}
		\Lambda^\eps(x)=(1-\varphi^\eps)(\ell_0+\omega_0\times x)+\nabla\varphi^\eps\times \Phi_0,
	\end{equation}
	where
	\begin{equation}
		\Phi_0(x)=-\frac12\,\ell_0\times x+\frac12|x|^2\omega_0.
	\end{equation}
	The approximating initial velocity is then given by
	\begin{equation}
		u_0^{\nu,\eps}(x)=\curl(\varphi^\eps\Psi)(x)+\Lambda^\eps(x),\qquad x\in\mathbb R^3.
	\end{equation}
	
	One readily verifies that $u_0^{\nu,\eps}$ satisfies \eqref{Initial-condition}. It remains to prove \eqref{Est:initial}. Indeed,
	\begin{equation}
		\|u_0^{\nu,\eps}-\widebar u_0\|_{L^2(\mathbb R^3)}
		\le \|(1-\varphi^\eps)\widebar u_0\|_{L^2(\mathbb R^3)}
		+\|\nabla\varphi^\eps\times\Psi+\Lambda^\eps\|_{L^2(\mathbb R^3)}.
	\end{equation}
	Since $1-\varphi^\eps$ is supported in $\overline{B(0,(R+2)\eps)}$ and $\widebar u_0$ is bounded, we have
	\begin{equation}\label{Est:initial-1}
		\|(1-\varphi^\eps)\widebar u_0\|_{L^2(\mathbb R^3)}\le K\eps^{3/2}.
	\end{equation}
	On the other hand, $\nabla\varphi^\eps$ is supported in the annulus $(R+1)\eps\le |x|\le (R+2)\eps$, where $|\nabla\varphi^\eps\times \Psi + \Lambda^\eps|\le K$. Hence
	\begin{equation}\label{Est:initial-2}
		\|\nabla\varphi^\eps\times\Psi+\Lambda^\eps\|_{L^2(\mathbb R^3)}\le K\eps^{3/2}.
	\end{equation}
	Combining \eqref{Est:initial-1} and \eqref{Est:initial-2} yields \eqref{Est:initial}.
\end{proof}

With $u_0^{\nu,\eps}$ constructed as above, we may combine \eqref{Est:tmp-100}-\eqref{Est:aa} and \eqref{Equ:diff-t-new-new} to obtain the following result.

\begin{corollary}\label{Cor:cor1}
	Let $\alpha$ and $\beta$ be as in Theorem \ref{Thm:main}. Then there exist constants $k=k(\Omega,\widebar u_0,T)>0$ and $K=K(\Omega,\widebar u_0,T,\rho_s)$ such that
	\begin{equation}
		\|u^{\nu,\eps}-\widebar u\|_{L^2(\mathbb R^3)}+\sqrt{E_{\mathrm{body}}^{\nu,\eps}(t)}
		\le K\Bigl(\eps^{\min(\alpha-1,3-\beta)/2}+\sqrt{\nu}\Bigr)
	\end{equation}
	for all $0<\eps\le k\nu$ and $0\le t\le T$.
\end{corollary}

Furthermore, when the initial linear velocity of the rigid body vanishes, Theorem \ref{Thm:endpoint} yields the following improvement.

\begin{corollary}\label{Cor:cor2}
	Suppose that $\ell_0=0$ and $\rho_s^{\nu,\eps}=\rho_s\eps^{-3}$ for some $\rho_s>0$. Then there exist constants $k=k(\Omega,\widebar u_0,T)>0$ and $K=K(\Omega,\widebar u_0,T,\rho_s)$ such that
	\begin{equation}
		\|u^{\nu,\eps}-\widebar u\|_{L^2(\mathbb R^3)}+\sqrt{E_{\mathrm{body}}^{\nu,\eps}(t)}
		\le K\sqrt{\nu}
	\end{equation}
	for all $0<\eps\le k\nu$ and $0\le t\le T$.
\end{corollary}
\begin{remark}
	Corollary \ref{Cor:cor1} and \ref{Cor:cor2} show that smallness of the initial kinetic energy of the rigid body is preserved for all subsequent times. This is consistent with the physical intuition that, as both the fluid viscosity and the body size tend to zero, the influence of the fluid on the rigid body motion becomes negligible.
\end{remark}

\section{Acknowledgments}
The work of Xiaoguang You was partially supported by the National Natural Science Foundation of China (Grant No.~12561038) and the Jiangxi Provincial Natural Science Foundation (Grant No.~20242BAB25008).

\section*{Data sharing}
Data sharing is not applicable to this article as no datasets were generated or analyzed during the current study.

\section*{Conflict of interest}
The authors declare that they have no financial or non-financial interests that are directly or indirectly related to the work submitted for publication.

\bibliographystyle{elsarticle-harv} 
\bibliography{my}
\end{document}